\documentclass[12pt,letterpaper]{article}
\usepackage[T1]{fontenc}
\usepackage{lmodern}
\usepackage[textwidth=5.45in,textheight=8.05in]{geometry}
\usepackage{amsmath,amssymb,amsthm}
\usepackage{microtype}
\usepackage[hidelinks]{hyperref}
\hypersetup{pdftitle={Rational linear forms of linearizable ordinary differential equations},pdfauthor={Dmitry A. Lyakhov},pdfsubject={Proposed shortened revision focused on rationality over the coefficient field}}
\newtheorem{theorem}{Theorem}
\newtheorem{lemma}[theorem]{Lemma}

\newcommand{\C}{\mathbb C}
\newcommand{\Q}{\mathbb Q}
\newcommand{\Z}{\mathbb Z}
\newcommand{\Span}{\operatorname{span}}

\title{Rational linear forms of linearizable\\ordinary differential equations}
\author{Dmitry A. Lyakhov}
\date{}
\begin{document}
\maketitle
\begin{abstract}
Every scalar ordinary differential equation of order at least three with rational right-hand side that is locally linearizable by a point transformation admits a linear form with rational coefficients over the same coefficient field. We prove this by restricting the derived symmetry algebra to a coordinate line and recovering a scalar differential operator from rational symmetry-jet data. The construction uses differential elimination and linear algebra; it does not require the symmetry generators or a linearizing transformation to be solved for. A single integer parameter suffices to choose the line.
\end{abstract}

\section{Statement and background}

Let $K\subset\C$ be a field and consider
\begin{equation}\label{eq:input}
 y^{(n)}+f(x,y,y',\ldots,y^{(n-1)})=0,
 \qquad f\in K(x,y,y',\ldots,y^{(n-1)}),\quad n\geq3.
\end{equation}
The equation is \emph{linearizable} if a locally invertible analytic point transformation $t=\psi(x,y)$, $u=\phi(x,y)$ maps it to
\begin{equation}\label{eq:linear}
 u^{(n)}+\sum_{k=0}^{n-1}a_k(t)u^{(k)}=0.
\end{equation}
All equivalences are local, away from singularities and where the prolonged transformation is defined.

\begin{theorem}\label{thm:main}
If \eqref{eq:input} is linearizable, it has a linear form
\begin{equation}\label{eq:output}
 w^{(n)}+\sum_{k=0}^{n-1}b_k(\lambda)w^{(k)}=0,
 \qquad b_k\in K(\lambda).
\end{equation}
If exact arithmetic and zero tests in $K$ are effective, such a form can be computed by differential elimination and linear algebra over rational function fields, without solving differential equations or constructing a linearizing transformation.
\end{theorem}

In particular, $K$ may be the field generated over $\Q$ by the finitely many coefficients of $f$. For the algorithmic assertion, an effective presentation of this field is required; an explicitly given finite algebraic extension of a rational function field over $\Q$ suffices. No extension of the constant field is needed. The theorem asserts the existence of a rational linear form, not rationality of every linear form or of the transformation itself.

Linearizability can be decided from the Lie point symmetry algebra $L$ \cite{MahomedLeach,LGM}. Writing $m=\dim_{\C}L$ and $L'=[L,L]$, the criterion for $n\geq3$ is
\begin{equation}\label{eq:criterion}
 m=n+4,\quad\text{or}\quad
 m\in\{n+1,n+2\},\quad \dim L'=n,\quad [L',L']=0.
\end{equation}
The maximal case has the linear form $u^{(n)}=0$. In the other two cases, in any linearizing coordinates,
\begin{equation}\label{eq:shifts}
 L'=\{v(t)\partial_u:\ v\text{ solves }\eqref{eq:linear}\}.
\end{equation}
Indeed, these $n$ independent symmetries belong to $L'$ because $[u\partial_u,v\partial_u]=-v\partial_u$; dimension then gives equality.

For $m=n+2$, recovery of a constant-coefficient form from the abstract symmetry algebra was studied in \cite{LM}. For $m=n+1$, the abstract algebra is always the semidirect product of an $n$-dimensional abelian ideal with a scalar action and cannot determine variable coefficients. The proof below instead uses its realization by vector fields, encoded by rational equations for their jets.

\section{Proof of the theorem}

\begin{lemma}\label{lem:jets}
Put $F=K(x,y)$. An involutive completion of the symmetry determining system yields a vector $q(X)$ of $m$ parametric derivatives for every $X=\xi\partial_x+\eta\partial_y\in L$, matrices $A_x,A_y\in F^{m\times m}$, rows $r_\xi,r_\eta\in F^{1\times m}$, and an $F$-bilinear map $B:F^m\times F^m\to F^m$, with
\begin{align}\label{eq:jets}
 \partial_xq(X)&=A_xq(X),& \partial_yq(X)&=A_yq(X),\notag\\
 \xi&=r_\xi q(X),&\eta&=r_\eta q(X),\\
 q([X,Y])&=B(q(X),q(Y)).&&\notag
\end{align}
On a nonempty algebraic open set, evaluation $X\mapsto q(X)(p)$ identifies symmetry germs with $\C^m$. Moreover,
\begin{equation}\label{eq:derived}
 E'=\Span_F\{B(e_i,e_j):1\leq i,j\leq m\}
\end{equation}
has fibre $q(L')(p)$ at every regular point.
\end{lemma}

\begin{proof}
The determining equations are linear over $F$. Involutive reduction expresses every derivative of $\xi,\eta$ as an $F$-linear combination of the parametric derivatives \cite{ReidA,ReidB,Seiler}. Reducing their first derivatives gives $A_x,A_y$; reducing $\xi,\eta$ gives the two rows. Reducing the derivatives of the Lie bracket gives $B$, since these derivatives are bilinear in the jets of $X,Y$.

On the regular set, the completed finite-type system is equivalent to the compatible first-order system in \eqref{eq:jets}. Its analytic solutions are uniquely determined by arbitrary initial vectors. This proves the evaluation assertion and, by taking brackets of arbitrary initial vectors, \eqref{eq:derived}.
\end{proof}

\begin{proof}[Proof of Theorem~\ref{thm:main}]
For $m=n+4$, return $w^{(n)}=0$. Otherwise \eqref{eq:criterion} gives $\dim_F E'=n$. Let the columns of $E\in F^{m\times n}$ form a basis of $E'$.

Choose a component and a differentiation direction by
\[
 (r,D,A)=
 \begin{cases}
 (r_\eta,\partial_x,A_x),&r_\eta E\ne0,\\
 (r_\xi,\partial_y,A_y),&r_\eta E=0.
 \end{cases}
\]
In the second case $r_\xi E\ne0$, since $L'$ is not the zero space of vector fields. Recursively define
\begin{equation}\label{eq:recurrence}
 R_0=r,\qquad R_{k+1}=D R_k+R_kA,
 \qquad
 M=\begin{pmatrix}R_0E\\ R_1E\\ \vdots\\ R_{n-1}E\end{pmatrix}.
\end{equation}
All these objects are rational over $K$.

We first prove $\det M\ne0$. In a chart of any given linearization, put $J=\psi_x\phi_y-\psi_y\phi_x\ne0$. The inverse image of $\partial_u$ is
\begin{equation}\label{eq:vertical}
 V=\frac{-\psi_y\partial_x+\psi_x\partial_y}{J},
 \qquad X_v=v(\psi)V\in L'.
\end{equation}
Thus $r_\eta E\ne0$ is equivalent to $\psi_x\not\equiv0$. In this case, on a local horizontal line, the selected components of a basis of $L'$ are
\[
 w_i(\lambda)=\mu(\lambda)v_i(T(\lambda)),\qquad
 T=\psi|_\ell,\quad \mu=(\psi_x/J)|_\ell.
\]
One can choose this local line so that $\mu T'\ne0$. If $r_\eta E=0$, then $\psi_x=0$ and $\psi_y\ne0$; use a vertical line and $\mu=(-\psi_y/J)|_\ell$ instead. In both cases,
\begin{equation}\label{eq:wronskian}
 W_\lambda(w_1,\ldots,w_n)
 =\mu^n(T')^{n(n-1)/2}W_t(v_1,\ldots,v_n)\circ T\ne0.
\end{equation}
By \eqref{eq:jets}, $D^k(rq(X))=R_kq(X)$. For a basis $X_1,\ldots,X_n$ of $L'$, write
\[
 (q(X_1),\ldots,q(X_n))=EC.
\]
The analytic matrix $C$ is invertible at regular points by Lemma~\ref{lem:jets}. Consequently, the Wronskian matrix is $MC$, so \eqref{eq:wronskian} proves that $M$ is invertible over $F$.

Define the rational row
\begin{equation}\label{eq:b}
 b=(b_0,\ldots,b_{n-1})=-R_nE\,M^{-1}\in F^{1\times n}.
\end{equation}
It gives the identity
\begin{equation}\label{eq:annihilator}
 D^n(rq(X))+\sum_{k=0}^{n-1}b_kD^k(rq(X))=0
 \qquad (X\in L').
\end{equation}
There is no differentiation of the moving coordinates $C$ in this step: all derivatives have already been computed in the ambient jet system by \eqref{eq:recurrence}.

Now specialize to $\ell:(x,y)=(\lambda,c)$ in the first case, or $\ell:(x,y)=(c,\lambda)$ in the second. Choose $c\in\Z$ so that no required denominator, regularity factor, a fixed nonzero rank minor of $E$, or the numerator of $\det M$ vanishes identically on $\ell$. Only finitely many $c$ are excluded: for a nonzero polynomial in the two variables, every bad $c$ is a root of a fixed nonzero coefficient polynomial in the specialized variable. Hence such an integer exists, and \eqref{eq:b} restricts to coefficients in $K(\lambda)$.

For clarity, equivalence can be verified near this chosen line even if it does not meet the initial linearizing chart. At an ordinary point of $\ell$, take local symmetry germs $X_1,\ldots,X_n$ spanning $L'$. They commute and have rank one as vector fields: the corresponding bracket identities and the $2\times2$ minors of the matrix with rows $r_\xi E,r_\eta E$ vanish in the initial chart, hence identically as rational functions. The line is transverse to this distribution, since the selected component is nonzero. The restricted components $w_i=(rq(X_i))|_\ell$ have a nonzero Wronskian by $W=\det M\det C$ and solve \eqref{eq:output} by \eqref{eq:annihilator}.

Straightening the distribution gives coordinates $(\lambda,z)$ with $\ell=\{z=0\}$ in which $X_i=w_i(\lambda)\partial_z$. To see the normalization, choose $w_1\ne0$, keep $\lambda$ constant on the orbits, and solve locally $X_1z=w_1(\lambda)$ with $z|_\ell=0$. Commutativity makes $X_i/X_1$ constant along the orbits, giving the stated formula for every $i$.

Write the transformed equation as $z^{(n)}=H(\lambda,z,\ldots,z^{(n-1)})$. Infinitesimal invariance under $w_i(\lambda)\partial_z$ gives
\[
 \sum_{k=0}^{n-1}w_i^{(k)}H_{z^{(k)}}=w_i^{(n)}
 =-\sum_{k=0}^{n-1}b_k|_\ell\,w_i^{(k)}.
\]
Invertibility of the Wronskian implies $H_{z^{(k)}}=-b_k|_\ell$. Thus the equation is linear, possibly inhomogeneous. Subtracting a local particular solution makes it exactly \eqref{eq:output}. These coordinate constructions are used only to prove equivalence; the computation consists of \eqref{eq:jets}--\eqref{eq:b} and rational specialization.
\end{proof}

\section{Computation and an example}

The proof gives the following procedure, including the linearizability test.
\begin{enumerate}
\item Complete the determining system to involution and determine $m$. If $m=n+4$, return $w^{(n)}=0$. If $m\notin\{n+1,n+2\}$, return \texttt{False}.
\item Construct the data of Lemma~\ref{lem:jets}. If $\dim E'\ne n$ or $B|_{E'\times E'}\ne0$, return \texttt{False}.
\item Choose $(r,D,A)$ as in the proof, form $M$ by \eqref{eq:recurrence}, and compute $b$ by \eqref{eq:b}.
\item Try $c=0,1,-1,2,-2,\ldots$ until the specialization conditions hold. Return \eqref{eq:output} with the specialized coefficients.
\end{enumerate}
Involutive completion terminates, the matrix computations are finite, and the last step excludes only finitely many integers. This proves termination under the effective-field hypothesis of Theorem~\ref{thm:main}. The search does not require an a priori degree bound. For efficiency, one may instead specialize the jet data first and perform the recurrence over $K(\lambda)$, rejecting singular choices by rank and determinant tests.

For a compact illustration, consider the polynomial equation over $\Q$
\begin{equation}\label{eq:example}
 y'''+3y'y''+(y')^3-x=0.
\end{equation}
The transformation $t=x$, $u=e^y$ gives $u'''=tu$. Its solution-translation symmetries pull back to
\[
 X_v=e^{-y}v(x)\partial_y,\qquad v'''=xv.
\]
On this ideal the rational reduced equations are $\eta_y=-\eta$ and $\eta_{xxx}=x\eta$. With $q=(\eta,\eta_x,\eta_{xx})^{\mathsf T}$, the $x$-connection is
\[
 \partial_xq=
 \begin{pmatrix}0&1&0\\0&0&1\\x&0&0\end{pmatrix}q.
\]
The component $\eta$ is cyclic: its first three rows are the coordinate rows, and the next row is $(x,0,0)$. Therefore \eqref{eq:b} gives $b=(-x,0,0)$. On $y=c$ the output is
\[
 w'''-\lambda w=0.
\]
The coefficients remain over $\Q$, although the displayed linearizing transformation is transcendental. No minimal-form construction or differential invariant is needed to recover this rational equation.


\begin{thebibliography}{9}
\bibitem{MahomedLeach}
F.~M. Mahomed and P.~G.~L. Leach.
Symmetry Lie algebras of $n$th order ordinary differential equations.
\emph{J. Math. Anal. Appl.} \textbf{151} (1990), 80--107.

\bibitem{LGM}
D.~A. Lyakhov, V.~P. Gerdt, and D.~L. Michels.
On the algorithmic linearizability of nonlinear ordinary differential equations.
\emph{J. Symbolic Comput.} \textbf{98} (2020), 3--22.

\bibitem{LM}
D.~A. Lyakhov and D.~L. Michels.
On the algorithmic recovering of coefficients in linearizable differential equations.
\href{https://arxiv.org/abs/2404.01798}{arXiv:2404.01798} (2024).

\bibitem{ReidA}
G.~J. Reid.
Algorithms for reducing a system of PDEs to standard form, determining the dimension of its solution space and calculating its Taylor series solution.
\emph{European J. Appl. Math.} \textbf{2} (1991), 293--318.

\bibitem{ReidB}
G.~J. Reid.
Finding abstract Lie symmetry algebras of differential equations without integrating determining equations.
\emph{European J. Appl. Math.} \textbf{2} (1991), 319--340.

\bibitem{Seiler}
W.~M. Seiler.
\emph{Involution: The Formal Theory of Differential Equations and its Applications in Computer Algebra}.
Springer, 2010.
\end{thebibliography}
\end{document}